\documentclass{amsart}
\usepackage{amssymb}
\usepackage{amscd}
\usepackage{amsfonts}
\usepackage{latexsym}
\usepackage[all]{xy}
\usepackage{hyperref}
\usepackage{verbatim}

\newtheorem{theorem}{Theorem}[section]
\newtheorem{lemma}[theorem]{Lemma}
\newtheorem{proposition}[theorem]{Proposition}
\newtheorem{corollary}[theorem]{Corollary} 
\theoremstyle{definition}  
\newtheorem{definition}[theorem]{Definition}
\newtheorem{example}[theorem]{Example}

\newtheorem{remark}[theorem]{Remark}

\newcommand{\id}{\text{id}}

\newcommand{\Fun}{\text{Fun}}

\newcommand{\Coh}{\text{Coh}}

\newcommand{\Irr}{\text{Irr}} 
\renewcommand{\Vec}{\text{Vec}}

\newcommand{\Hom}{\text{Hom}}

\newcommand{\Ind}{\text{Ind}}

\newcommand{\C}{\mathcal{C}}
\newcommand{\cP}{\mathcal{P}}
\newcommand{\D}{\mathcal{D}}

\newcommand{\F}{\mathcal{F}}
\newcommand{\Z}{\mathcal{Z}}

\newcommand{\M}{\mathcal{M}}
\newcommand{\A}{\mathcal{A}}

\newcommand{\K}{\mathbb{K}}

\newcommand{\be}{\mathbf{1}}

\renewcommand{\be}{\mathbf{1}}

\newcommand{\BZ}{{\mathbb Z}}
\newcommand{\BC}{{\mathbb C}}
\newcommand{\BQ}{{\mathbb Q}}

\newcommand{\bt}{\boxtimes}
\newcommand{\bc}{{\bf c}}
\newcommand{\ot}{\otimes}
\newcommand{\st}{\circledast}

\begin{document}

\title{Tensor categories attached to exceptional cells in Weyl groups}

\author{Victor Ostrik}
\address{Department of Mathematics,
University of Oregon, Eugene, OR 97403, USA}
\email{vostrik@math.uoregon.edu}

\begin{abstract}
Using truncated convolution of perverse sheaves on a flag variety 
Lusztig associated a monoidal category to a two  sided cell in the Weyl group.
We identify this category in the case which was not decided previously.
\end{abstract}

\date{\today} 
\maketitle  

\section{Introduction}
Let $W$ be a finite irreducible crystallographic Coxeter group and let $\Irr(W)$ be the
set of isomorphism classes of irreducible representations of $W$ over $\BQ$.
It is well known that two irreducible representations of dimension 512 for $W$ of type $E_7$ 
and four irreducible representations of dimension 4096 for $W$ of type $E_8$ behave differently
from other elements of $\sqcup_W\Irr(W)$ in many ways, see e.g. \cite{Cu}, \cite{BL}, 
\cite[Theorem 4.23]{Lb} (see \cite[p. 109]{Lb} for the definition of function $\Delta$). 
For this reason the representations above are called {\em exceptional}.
In this note we give one more example of unusual behavior of exceptional representations
or rather of associated geometric objects.

Recall that the group $W$ is partitioned into {\em two sided cells}, $W=\sqcup \bc$, see e.g. 
\cite[p. 137]{Lb}. 
We say that a two sided cell $\bc$ is {\em exceptional} if the corresponding {\em family} 
(see \cite[5.15, 5.25]{Lb}) consists of exceptional representations of $W$. 
Thus there are just three exceptional two sided cells,
one for $W$ of type $E_7$ and two for $W$ of type $E_8$, see \cite[Chapter 4]{Lb}. 

We fix an algebraically closed field $\K$ of characteristic zero.
Let $G$ be a simple algebraic group over $\BC$ with the Weyl group $W$. 
Using truncated convolution of perverse $\K-$sheaves (in the classical topology) 
on the flag variety of $G$,
Lusztig associated to each two sided cell $\bc \subset W$
a semisimple monoidal category $\cP^{\bc}$ (over $\K$) , see \cite{Lt} and 
\S \ref{trunk} below. 
Moreover, Lusztig conjectured that there is a tensor equivalence 
$\cP^\bc \simeq \Coh_{A(\bc)}(Y\times Y)$ where $A(\bc)$ is a finite group associated
with family corresponding to $\bc$ in \cite[Chapter 4]{Lb}, $Y$ is a finite $A(\bc)$-set
and $\Coh_{A(\bc)}(Y\times Y)$ is the category of $A(\bc)$-equivariant
 sheaves on the set $Y\times Y$ with convolution as a tensor product and a natural 
 associativity constraint, see \cite[\S 3.2]{Lt}.
This conjecture was verified in \cite{BFO} for all non-exceptional two sided cells.

Let $\bc$ be an exceptional cell. Then $A(\bc)=\BZ/2\BZ$ is the cyclic group of order 2. 
The Lusztig's conjecture from \cite{Lt} predicts that we have a tensor equivalence
$\cP^{\bc}\simeq \Coh_{A(\bc)}(Y\times Y)$ where $Y$ is a finite set with {\em free} 
$A(\bc)=\BZ/2\BZ$-action (the set $Y$ should be of cardinality 1024 if $W$ is of type $E_7$
 and of cardinality 8192 if $W$ is of type $E_8$). Our main goal
is to show that one needs to change the associativity constraint a little bit in order to make this
statement correct.

Let $Y'$ be a finite set of cardinality 512 if $W$ is of type $E_7$ and of cardinality 4096
if $W$ is of type $E_8$ (one can identify $Y'$ with
the set of $\BZ/2\BZ$-orbits on the set $Y$). Let $\Vec_{\BZ/2\BZ}$ be the
monoidal category of finite dimensional $\BZ/2\BZ$-graded spaces. Then there is a tensor
equivalence $\Coh_{\BZ/2\BZ}(Y\times Y)\simeq \Vec_{\BZ/2\BZ}\bt \Coh(Y'\times Y')$.
The category $\Vec_{\BZ/2\BZ}$ has two simple objects, the unit object $\be$ and one more
simple object $\delta$. Let $\Vec_{\BZ/2\BZ}^\omega$ be the same as category $\Vec_{\BZ/2\BZ}$
but with modified associativity constraint: for simple objects $X, Y, Z$ the associativity
constraint $(X\ot Y)\ot Z\simeq X\ot (Y\ot Z)$ is the same as in the category $\Vec_{\BZ/2\BZ}$
if at least one of $X, Y, Z$ is isomorphic to $\be$ and the associativity constraint
$(\delta \ot \delta)\ot \delta \simeq \delta \ot (\delta \ot \delta)$ in $\Vec_{\BZ/2\BZ}^\omega$ 
differs by sign from the one in $\Vec_{\BZ/2\BZ}$. Here is our main result:

\begin{theorem} \label{main}
For an exceptional two sided cell $\bc$ there is a tensor equivalence $\cP^{\bc}\simeq 
\Vec_{\BZ/2\BZ}^\omega \bt \Coh(Y'\times Y')$.
\end{theorem}

\begin{remark} (i) It will be clear from the proof that the category 
$\Vec_{\BZ/2\BZ}^\omega \bt \Coh(Y'\times Y')$ is not equivalent to
$\Vec_{\BZ/2\BZ} \bt \Coh(Y'\times Y')$.

(ii) Using the methods from \cite{BFO} one can show that for an exceptional two sided cell 
$\bc$ one has
either $\cP^{\bc}\simeq \Vec_{\BZ/2\BZ} \bt \Coh(Y'\times Y')$ or
$\cP^{\bc}\simeq \Vec_{\BZ/2\BZ}^\omega \bt \Coh(Y'\times Y')$. Thus our main
result is just a computation of the sign in the associativity constraint.

(iii) Lusztig's construction of the category $\cP^\bc$ and our arguments extend with trivial
changes to the settings of $D$-modules ($G$ is defined over an algebraically closed 
field of characteristic zero) or $\ell$-adic sheaves ($G$ defined over an
algebraically closed field in which $\ell$ is invertible).

(iv) There is an alternative definition of the category $\cP^\bc$
based on idea of A.~Joseph. Namely, it is proved in \cite[Corollary 4.5(b)]{BFO2} that 
($D$-module counterpart of) the category $\cP^\bc$ is tensor equivalent to certain 
subquotient of the category of Harish-Chandra bimodules with tensor product coming
from the usual tensor product of bimodules.

(v) The arguments in the proof of Theorem \ref{main} extend to the case of monodromic 
sheaves on the flag variety as in \cite[\S 5]{BFO}. Thus we get a proof of statement in
\cite[Remark 5.7]{BFO}.
\end{remark}

A direct computation of the associativity constraint in the category $\cP^\bc$ seems to be
difficult. Thus our proof of Theorem \ref{main} is indirect. Our main tool is the theory of
(unipotent) {\em character sheaves} developed by Lusztig \cite{Lu}. We use 
a {\em commutator functor} (see \cite[\S 6]{BFO}) from $\cP^\bc$ to the category of sheaves 
on the group $G$ (values of this functor are direct sums of character sheaves on $G$).
Our main observation is that a commutator functor carries a canonical automorphism and
keeping track of the order of this automorphism is 
sufficient in order to distinguish between categories  
$\Vec_{\BZ/2\BZ}$ and $\Vec_{\BZ/2\BZ}^\omega$.

\section{Proofs}

\subsection{Commutator functors}
Let $\C$ be a monoidal category with associativity isomorphism $a$ and let $\A$ be a category.

\begin{definition}{\em (\cite[\S 6]{BFO})} A {\em commutator functor} $F: \C \to \A$ is a functor
endowed with a natural isomorphism $u_{X,Y}: F(X\ot Y)\simeq F(Y\ot X)$ such that the following  
diagram commutes:
$$
\xymatrix{
&F(X\ot (Y\ot Z))\ar^{u_{X,Y\ot Z}}[dr]&\\
F((X\ot Y)\ot Z)\ar^{F(a_{X,Y,Z})}[ur]\ar^{u_{X\ot Y,Z}}[d]&&F((Y\ot Z)\ot X))\ar^{F(a_{Y,Z,X})}[d]\\
F(Z\ot (X\ot Y))\ar^{F(a^{-1}_{Z,X,Y})}[dr]&&F(Y\ot (Z\ot X))\\
&F((Z\ot X)\ot Y)\ar^{u_{Z\ot X,Y}}[ur]&}
$$
\end{definition}

\begin{remark} It follows immediately from definition that $u_{X,\be}: F(X)\to F(X)$ is an idempotent
and hence an identity map. 
\end{remark} 

\begin{definition} The map $u_{\be,X}: F(X)\to F(X)$ is called {\em canonical automorphism}
of the commutator functor $F$.
\end{definition}

\begin{remark} \label{twice}
It follows from definition that $u_{Y,X}\circ u_{X,Y}=u_{\be, X\ot Y}$.
\end{remark}

Recall that a {\em central functor} $G: \A \to \C$ is a functor endowed with functorial isomorphism
$v_{A,X}: G(A)\ot X\simeq X\ot G(A)$ such that the following diagram commutes:
$$
\xymatrix{
&G(A)\ot (X\ot Y)\ar^{v_{A,X\ot Y}}[dr]&\\
(G(A)\ot X)\ot Y\ar^{a_{G(A),X,Y}}[ur]\ar^{v_{A,X}\ot \id_Y}[d]&&(X\ot Y)\ot G(A)\ar^{a_{X,Y,G(A)}}[d]\\
(X\ot G(A))\ot Y\ar^{a_{X,G(A),Y}}[dr]&&X\ot (Y\ot G(A))\\
&X\ot (G(A)\ot Y)\ar^{\id_X\ot v_{A,Y}}[ur]&}
$$

Equivalently, the structure of central functor on $G$ is the same as factorization
$\A \to \Z(\C)\to \C$ where $\Z(\C)$ is the Drinfeld center of $\C$ and $\Z(\C)\to \C$ is the
forgetful functor, see e.g. \cite[\S 2.3]{ENO}.

\begin{proposition} \label{commcent} {\em (\cite[\S 6]{BFO})}
Assume that the category $\C$ is rigid and has a pivotal structure (that is an isomorphism
of tensor functors $\mbox{Id} \to \phantom{}^{**}$). Let $(G,F)$ be an adjoint pair of
functors between $\C$ and $\A$. Then the structures of commutator functor on $F$ are
in natural bijection with structures of central functor on $G$.
\end{proposition}

\begin{proof} Assume that $G$ has a central structure. Then
$$
\Hom(?,F(X\ot Y))=\Hom(G(?),X\ot Y)=\Hom(G(?)\ot \phantom{}^*Y,X)
$$
$$
=\Hom(\phantom{}^*Y\ot G(?),X)=\Hom(G(?), \phantom{}^{**}Y\ot X)=\Hom(?, F(Y\ot X))
$$
and it is straightforward to check that the resulting isomorphism $F(X\ot Y)\simeq F(Y\ot X)$
satisfies the definition of commutator functor.

Conversely, assume that $F$ has a commutator structure. Then 
$$\Hom(G(A)\ot X,?)=\Hom(G(A), ?\ot X^*)=\Hom(A, F(?\ot X^*))$$
$$=\Hom(A,F(X^*\ot?))=\Hom(G(A), X^*\ot ?)=\Hom(X\ot G(A), ?)$$
and it is straightforward to check that the resulting isomorphism $G(A)\ot X\simeq X\ot G(A)$
satisfies the definition of central functor.

Finally, one verifies that the two constructions above are mutually inverse.
\end{proof}

Let $\C$ be a pivotal category.
Assume that the forgetful functor $\Z(\C)\to \C$ has a right adjoint functor $\Ind: \C \to \Z(\C)$.
It follows from Proposition \ref{commcent} that the functor $\Ind$ has a canonical structure 
of commutator functor. Clearly, for any functor $\Z(\C)\to \A$ the composition 
$\C \xrightarrow{\Ind} \Z(\C)\to \A$ has a structure of commutator functor. Moreover, we have the following
universal property:

\begin{corollary} Let $F: \C \to \A$ be a commutator functor such that left adjoint of $F$ exists.
Then $F$ factorizes as $\C \xrightarrow{\Ind} \Z(\C)\to \A$. $\square$
\end{corollary}

Let $\D$ be a ribbon category, for example Drinfeld center $\Z(\C)$ of a spherical category $\C$.
Recall that there is a canonical automorphism $\theta$ of the identity functor called {\em twist}
defined as following composition:
$$ X\to \phantom{}^{**}X\ot \phantom{}^*X\ot X\to \phantom{}^{**}X\ot X\ot \phantom{}^*X\to 
\phantom{}^{**}X=X.$$
One can rewrite this definition in terms of functor represented by $X$ as follows:
$$\Hom(?,X)=\Hom(?\ot \phantom{}^*X, \be)=\Hom(\phantom{}^*X\ot ?,\be)=
\Hom(?, \phantom{}^{**}X).$$
Thus we have the following

\begin{corollary} The canonical automorphism of the commutator functor $\Ind: \C \to \Z(\C)$
equals the twist automorphism. $\square$
\end{corollary}

\begin{example} \label{C2}
(a) Let $\C =\Vec_{\BZ/2\BZ}$. It is well known that the category $\C$ is spherical
(in two different ways) and the twist of any simple object of $\Z(\C)$ is $\pm 1$ (for any choice
of the spherical structure). It follows that for any
commutator functor $F: \C \to \A$ such that left adjoint of $F$ exists the square of the canonical
automorphism is the identity map.

(b) Let $\C =\Vec_{\BZ/2\BZ}^\omega$. The category $\C$ has two spherical structures and 
the twists of simple objects of $\Z(\C)$ are 1 and $\pm \sqrt{-1}$. Thus there exists a
commutator functor $F: \C \to \A$ such that left adjoint of $F$ exists and the square of the canonical
automorphism is not the identity map, for example $F=\Ind$.
\end{example}

\subsection{Commutator functor $\Gamma$ and its canonical automorphism} Let $G$ be a simple algebraic group over an algebraically closed 
field and let $B \subset G$ be a Borel subgroup. The group $B\times B$ acts on $G$ via left and
right translations. We consider the subcategories $\C_{B\times B}$ and $\C_G$ of the bounded derived category of constructible sheaves on $G$ consisting of complexes that are isomorphic 
to a finite direct sum of shifts of simple, $B\times B$-equivariant (respectively, $G$-equivariant with respect to the adjoint action) perverse sheaves on $G$.

We recall now that
the category $\C_{B\times B}$ has a natural monoidal structure with tensor product given by
the convolution, see e.g. \cite[Section 1]{Lt}. The group $B$ acts freely on 
$G\times G$ as follows: $b\cdot (x,y)=
(xb,b^{-1}y)$. Let $G\times_BG$ denote the quotient of $G\times G$ by this action and let
$\pi: G\times G\to G\times_BG$ be the natural projection. Clearly, the multiplication map
$m(x,y)=xy$ is a well defined map $m:G\times_BG\to G$. For $\F_1, \F_2\in \C_{B\times B}$
let $\widetilde{\F_1\bt \F_2}$ be the unique complex of sheaves on $G\times_BG$ such that
$\pi^*\widetilde{\F_1\bt \F_2}=\F_1\bt \F_2$. By definition, the convolution of $\F_1$ and $\F_2$
is $\F_1*\F_2:=m_!\widetilde{\F_1\bt \F_2}$. It follows from the Decomposition Theorem 
(see \cite{BBD}) that
$\F_1*\F_2\in \C_{B\times B}$ and it is clear that the convolution is a bifunctor. 

Let  $G\times_BG\times_BG$ be the
quotient of $G\times G \times G$ by the free $B\times B$-action via $(b_1,b_2)\cdot (x,y,z)=(xb_1,b_1^{-1}yb_2,b_2^{-1}z)$ with obvious projection $\pi_2: G\times G\times G\to G\times_BG\times_BG$ and multiplication map $m_2: G\times_BG\times_BG\to G$, $m_2(x,y,z)=xyz$.
For $\F_1, \F_2, \F_3\in \C_{B\times B}$ the convolutions $(\F_1*\F_2)*\F_3$ and $\F_1*(\F_2*\F_3)$ are both
canonically isomorphic to $(m_2)_!(\widetilde{\F_1\bt \F_2\bt \F_3})$ where 
$\widetilde{\F_1\bt \F_2\bt \F_3}$ is the unique complex on $G\times_BG\times_BG$
such that $\pi_2^*(\widetilde{\F_1\bt \F_2\bt \F_3})=\F_1\bt \F_2\bt \F_3$. 
Identifying 
$(\F_1*\F_2)*\F_3$ and $\F_1*(\F_2*\F_3)$ via these isomorphisms we endow the category
$\C_{B\times B}$ with associativity constraint satisfying the pentagon axiom. Moreover, 
$\C_{B\times B}$ is a monoidal category with unit object given by the constant sheaf on 
$B\subset G$.

\begin{remark} The indecomposable objects of the category $\C_{B\times B}$ are
$IC_w[i]$ where $IC_w$ is the intersection cohomology complex (see \cite{BBD}) 
of the Bruhat cell corresponding to $w\in W$ and $[i]$ stands for the shift.
It is well known (see e.g. \cite[\S 1.4]{Lt}) that the (split) Grothendieck ring of the category 
$\C_{B\times B}$ identifies with {\em Hecke algebra}; under this identification
the objects $IC_w[i]$ correspond to the elements of Kazhdan-Lusztig basis multiplied
by $i$-th power of the Hecke algebra parameter.
\end{remark}

There is an equivariantization functor $\Gamma: \C_{B\times B}\to \C_G$ defined as follows
(see \cite[Section 1]{MV}).
Let $G\tilde \times_BG$ be the quotient of $G\times G$ by the following free $B$-action:
$b\odot (g,x)=(b^{-1}g,b^{-1}xb)$. We have a canonical projection 
$\tilde \pi: G\times G\to G\tilde \times_BG$ and
the adjoint action map $a: G\tilde \times_BG\to G$, $a(g,x)=g^{-1}xg$.
For $\F \in \C_{B\times B}$ let $\bar \F$ be a unique complex on $G\tilde \times_BG$ such that
$\tilde \pi^*\bar \F=K\bt \F$. We set $\Gamma(\F):=a_!(\bar \F)$. It is clear that $\Gamma(\F)$ is
$G$-equivariant complex on $G$ (with respect to the adjoint action) and the Decomposition 
Theorem (see \cite{BBD}) implies that $\Gamma(\F)$ is semisimple. In other words, 
$\Gamma$ is a functor $\C_{B\times B}\to \C_G$.

\begin{remark} \label{chsh}
We recall that irreducible constituents of perverse cohomology of complexes
$\Gamma(\F)$ are by definition {\em character sheaves} (with trivial central character, or unipotent)
on $G$, see \cite[Sections 2, 11]{Lu} and \cite[Lemma 2.3]{MV}.
\end{remark}

We now define commutator structure on the functor $\Gamma$. Let $G\tilde \times_B(G\times_BG)$
be the quotient of $G\times G\times G$ with respect to the following free $B\times B$-action:
$(b_1,b_2)\cdot (g,x,y)=(b_1^{-1}g, b_1^{-1}xb_2, b_2^{-1}yb_1)$. We have an obvious
projection $\tilde \pi_2: G\times G\times G\to G\tilde \times_B(G\times_BG)$ and a well defined
map $a_2: G\tilde \times_B(G\times_BG)\to G$, $a_2(g,x,y)=g^{-1}xyg$. We have an obvious
isomorphism $\Gamma(\F_1*\F_2)=(a_2)_!(\overline{\F_1\bt \F_2})$ where 
$\overline{\F_1\bt \F_2}$ is a unique complex on $G\tilde \times_B(G\times_BG)$ such that
$\tilde \pi_2^*(\overline{\F_1\bt \F_2})=K\bt \F_1\bt \F_2$.

Consider the following (well defined) map 
$\rho: G\tilde \times_B(G\times_BG)\to G\tilde \times_B(G\times_BG)$, $\rho(g,x,y)=(yg,y,x)$.
For $\F_1, \F_2\in \C_{B\times B}$ there is a unique isomorphism $\tilde u_{\F_1,\F_2}:
\rho_!(\overline{\F_1\bt \F_2})\simeq \overline{\F_2\bt \F_1}$ inducing the identity map 
in every stalk (clearly, the stalks of both $\rho_!(\overline{\F_1\bt \F_2})$ and 
$\overline{\F_2\bt \F_1}$ at $(g,x,y)\in G\tilde \times_B(G\times_BG)$ are canonically
isomorphic to $(\F_2)_x\ot (\F_1)_y$). Now observe that $a_2\circ \rho =a_2$. Thus
we have an isomorphism $u_{\F_1, \F_2}: \Gamma(\F_1*\F_2)\simeq \Gamma(\F_2*\F_1)$
defined as composition
$$\Gamma(\F_1*\F_2)=(a_2)_!(\overline{\F_1\bt \F_2})=(a_2\circ \rho)_!(\overline{\F_1\bt \F_2})
\xrightarrow{(a_2)_!(\tilde u_{\F_1, \F_2})} (a_2)_!(\overline{\F_2\bt \F_1})=\Gamma(\F_2*\F_1).$$ 

The following result is straightforward:

\begin{proposition} The functor $\Gamma$ together with isomorphism $u_{\bullet,\bullet}$ is
a commutator functor $\C_{B\times B}\to \C_G$. $\square$
\end{proposition}

There is a canonical automorphism $\Theta$ of the identity functor of the category $\C_G$ defined 
as follows (see \cite[Definition 2.1]{Ocusp}): for any $\F \in \C_G$ by definition of equivariance 
there is an isomorphism $ad^*(\F)\simeq p^*(\F)$ where $p: G\times G\to G$ is the second 
projection and $ad: G\times G\to G$ is the adjoint action map, $ad(g,x)=g^{-1}xg$; taking
pullback of this isomorphism with respect to the diagonal map $\Delta: G\to G\times G$ and
noticing that $\Delta \circ ad=\Delta \circ p=\text{Id}$ we get an automorphism $\Theta_\F :\F \to \F$.
In other words, the automorphism $\Theta_\F$ at the stalk $\F_x$ is precisely $ad(x): \F_x\simeq
\F_{x^{-1}xx}$. We have the following 

\begin{proposition} \label{theta}
The canonical automorphism of the commutator functor $\Gamma$ equals $\Theta$, that is
$u_{\be,\F}=\Theta_{\Gamma(\F)}$. $\square$
\end{proposition}

\begin{remark} Our construction of commutator structure on the functor $\Gamma$ is just a
more explicit version of construction in \cite[\S 6]{BFO} (we will not need this fact in what follows).
The advantage of the present version is that it makes Proposition \ref{theta} easy.
\end{remark}

\subsection{Truncated convolution and equivariantization} \label{trunk}
We recall (see e.g. \cite[Chapter 5]{Lb}) that the Weyl group $W$ is partitioned into 
two sided cells, $W=\sqcup \bc$.
Remind that there is a partial order $\le_{LR}$ on the set of two sided cells.
Let $a(\bc)$ denote the common value of Lusztig's $a$-function on $w\in W$, see \cite[\S 2.3]{Lt}.

Let $\cP_{B\times B}$ denote the full subcategory of $\C_{B\times B}$
consisting of perverse sheaves.
Let $\cP_{B\times B}^{\bc}$ denote the full subcategory of $\cP_{B\times B}$ consisting of direct
sums of perverse sheaves $IC_w$, $w\in \bc$ and let $pr^{\bc}: \cP_{B\times B}\to \cP_{B\times B}^{\bc}$ be the obvious projection functor. Let $\phantom{}^pH$ denote the perverse cohomology 
functor. Consider the following bifunctor
$\cP_{B\times B}^{\bc}\times \cP_{B\times B}^{\bc}\to \cP_{B\times B}^{\bc}$:
$$\F_1\st \F_2:=pr^{\bc}(\phantom{}^pH^{a(\bc)}(\F_1*\F_2)).$$
It is explained in \cite[Section 2]{Lt} that the associativity constraint of the convolution category $\C_{B\times B}$
restricts to the associativity constraint for the bifunctor $\st$. Moreover, there exists the unit
object $\be_{\bc} \in \cP_{B\times B}^{\bc}$ (see \cite[\S 2.9]{Lt}), so the category 
$\cP_{B\times B}^{\bc}$ has a monoidal
structure. We have the following result

\begin{proposition} {\em (\cite[p. 222]{BFO})}
The monoidal category $\cP_{B\times B}^{\bc}$ is rigid. Hence $\cP_{B\times B}^{\bc}$ is a
multi-fusion category in a sense of \cite[\S 2.4]{ENO}. $\square$
\end{proposition}

\begin{remark} \label{J}
 The Grothendieck ring of the category $\cP_{B\times B}^{\bc}$ identifies
with Lusztig's {\em asymptotic Hecke ring} $J_\bc$, see \cite[\S 2.6]{Lt}. This together with
\cite[Corollary 12.16]{Lb} implies that the multi-fusion category $\cP_{B\times B}^{\bc}$ is 
{\em indecomposable} in the sense of \cite[\S 2.4]{ENO}.
\end{remark}

Let $\cP_G$ be the full subcategory of $\C_G$ consisting of direct sums of (unshifted) unipotent
character sheaves, see Remark \ref{chsh}. It is known (see \cite[Section 16]{Lu}, \cite[\S 3.4]{Gr}) 
that there is a unique direct sum decomposition
(or, equivalently, partition of the set of isomorphism classes of unipotent character sheaves)
$\cP_G=\bigoplus_{\bc}\cP_G^{\bc}$ with the following properties:

(i) For $w\in \bc$ and $i\in \BZ$ we have $\phantom{}^pH^i(\Gamma(IC_w))\in \bigoplus_{\bc'\le_{LR}\bc}\cP_G^{\bc'}$; moreover $\phantom{}^pH^i(\Gamma(IC_w))\in \bigoplus_{\bc'<_{LR}\bc}\cP_G^{\bc'}$ if $|i|>a(\bc)$.

(ii) For a simple perverse sheaf $A\in \cP_G^{\bc}$ there exists $w\in \bc$ such that
$A$ is a direct summand of $\phantom{}^pH^{a(\bc)}(\Gamma(IC_w))$.

Let $\widetilde{pr}^{\bc}$ denote the projection functor $\cP_G\to \cP_G^{\bc}$.
We consider the functor $\Gamma^{\bc}: \cP_{B\times B}^{\bc}\to \cP_G^{\bc}$,
$\Gamma^{\bc}(\F):=\widetilde{pr}^{\bc}(\phantom{}^pH^{a(\bc)}(\Gamma(\F))$.
Notice that the property (ii) implies that the functor $\Gamma^{\bc}$ is surjective,
that is each simple object of $\cP_G^{\bc}$ appears as a direct summand of some
$\Gamma^{\bc}(\F)$.
The properties above imply that the restriction of the commutator structure of the
functor $\Gamma$ gives a well defined commutator structure $u^\bc$ 
of the functor $\Gamma^{\bc}$, see \cite[\S 6]{BFO}. We have

\begin{proposition} \label{gamc}
The functor $\Gamma^{\bc}: \cP_{B\times B}^{\bc}\to \cP_G^{\bc}$ is a surjective commutator
functor with canonical automorphism $\Theta$.
\end{proposition}

\begin{proof} We need only check the statement about the canonical automorphism
of $\Gamma^\bc$. By definition the object $\Gamma^\bc(\be_\bc \st \F)$ is a direct summand of 
$\Gamma(\be_\bc *\F)$. The canonical automorphism of 
$\Gamma^\bc(\F)=\Gamma^\bc(\be_\bc \st \F)$ equals 
the restriction of automorphism $u_{\F,\be_\bc}\circ u_{\be_\bc ,\F}$ of $\Gamma(\be_\bc *\F)$ to
this direct summand. By Remark \ref{twice} the composition $u_{\F,\be_\bc}\circ u_{\be_\bc ,\F}$
equals the canonical automorphism of $\Gamma(\be_\bc *\F)$. Now the result follows from
Proposition \ref{theta}.
\end{proof}

\subsection{Proof of Theorem \ref{main}}
Recall that $\bc$ is an exceptional two sided cell. Let $\cP^\bc =\cP^\bc_{B\times B}$.

The following result follows easily from \cite[p. 347]{Sh2} (for $W$ of type $E_7$) and 
\cite[Proposition 3.6]{Sh1} (for $W$ of type $E_8$):

\begin{proposition} Assume that $\bc$ is an exceptional two sided cell. The category $\cP_G^\bc$
has 4 isomorphism classes of simple objects. The automorphism $\Theta$ acts as {\em Id} on 
two of them and as {\em $\pm \sqrt{-1}\mbox{Id}$} on two others. $\square$
\end{proposition}

It follows from Proposition \ref{gamc} that the canonical automorphism of the commutator functor
$\Gamma^\bc : \cP^{\bc}\to \cP_G^{\bc}$ has order 4.

Let $\be_\bc=\oplus_i\be_i$ be the decomposition of the unit object $\be_\bc \in \cP^\bc$ into
irreducible summands. It is clear that $\{ \be_i\}_{i}$ are orthogonal idempotents, that is
$\be_i\ot \be_i\simeq \be_i$ and $\be_i\ot \be_j=0$ if $i\ne j$. For any simple object $X\in \cP^\bc$
there are unique $i$ and $j$ such that $X=\be_i\ot X\ot \be_j$; observe that $i\ne j$ implies
$\Gamma^\bc(X)=0$ (indeed, in this case $\Gamma^\bc(X)=\Gamma^\bc(\be_i\ot X)=
\Gamma^\bc(X\ot \be_i)=\Gamma^\bc(0)=0$). Thus there exists $i$ and a simple object 
$X=\be_i\ot X\ot \be_i\in \cP^\bc$
such that $\Gamma^\bc(X)$ contains an irreducible summand $\F$ with 
$\Theta_\F=\pm \sqrt{-1}\mbox{Id}_\F$. We fix such $i$.

Clearly $\be_i\ot \cP^\bc \ot \be_i\subset \cP^\bc$ is a tensor
subcategory with unit object $\be_i$ and the restriction $\Gamma^\bc|_{\be_i \ot \cP^\bc \ot \be_i}$
is a commutator functor. The choice of $i$ guarantees that the order of the canonical automorphism
of $\Gamma^\bc|_{\be_i \ot \cP^\bc \ot \be_i}$ is 4. 
The Grothendieck ring of the category $\be_i\ot \cP^\bc \ot \be_i$
is computed in \cite[Corollary 3.12]{Ll} (cf. Remark \ref{J}); we see that this category has
precisely 2 isomorphism classes of simple objects $\be_i$ and $\delta_i$, furthermore
$\delta_i\ot \delta_i\simeq \be_i$. The following result is well known; it is a special
case of results of \cite{S}:

\begin{lemma}
Let $\C$ be a fusion category with two (isomorphism classes of) simple objects: 
the unit object $\be$ and one more object $\delta$ such that $\delta \ot \delta\simeq \be$.
Then either $\C \simeq \Vec_{\BZ/2\BZ}$ or  $\C \simeq \Vec_{\BZ/2\BZ}^\omega$. $\square$
\end{lemma}

In view of Example \ref{C2} we see that $\be_i\ot \cP^\bc \ot \be_i$ can not
be equivalent to $\Vec_{\BZ/2\BZ}$ (notice that the commutator functor 
$\Gamma^\bc|_{\be_i \ot \cP^\bc \ot \be_i}$ has the left adjoint since it is a functor
between semisimple categories). Hence 
$\be_i\ot \cP^\bc \ot \be_i\simeq \Vec_{\BZ/2\BZ}^\omega$.

Let $\C$ be a multi-fusion category and let $e\in \C$ be a direct summand of the unit object.
Then $e\ot \C \ot e$ is a tensor subcategory of $\C$ and $e\ot \C$ is a {\em module category}
over $e\ot \C \ot e$ via the left multiplication, see \cite{O1}. Each object $X\in \C$ gives rise
to a functor $?\ot X: e\ot \C \to e\ot \C$ commuting with the module structure above. 
Thus we get a tensor functor from $\C$ to the category $\Fun_{e\ot \C \ot e}(e\ot \C,e\ot \C)$
of $e\ot \C \ot e$-module endofunctors of $e\ot \C$. The following result is well known:

\begin{lemma} Assume that multi-fusion category $\C$ is indecomposable.
The functor above is an equivalence $\C \simeq \Fun_{e\ot \C \ot e}(e\ot \C,e\ot \C)$.
\end{lemma}

\begin{proof} The category $e\ot \C$ is a right module category over $\C$ and the obvious 
functor $e\ot \C \ot e\to \Fun_\C(e\ot \C,e\ot \C)$ is an equivalence with the inverse functor
$F\mapsto F(e)$. In the language of \cite[\S 4.2]{O1}, this means that the category
$e\ot \C \ot e$ is dual to the category $\C$. Now, the result follows by duality, see 
\cite[Corollary 4.1]{O1}.
\end{proof}

Thus we can reconstruct multi-fusion category $\C$ from smaller category $e\ot \C \ot e$ 
and module category $e\ot \C$. We apply this to the case $\C =\cP^\bc$ and $e=\be_i$. 
It is well known (see e.g. \cite[Remark 6.2(iii)]{O1}) that the fusion category 
$\be_i\ot \cP^\bc \ot \be_i\simeq \Vec_{\BZ/2\BZ}^\omega$ has only one 
indecomposable semisimple module category, namely the regular
module category $\Vec_{\BZ/2\BZ}^\omega$. Hence $\cP^\bc \simeq \Fun_{\Vec_{\BZ/2\BZ}^\omega}(\M,\M)$ where $\M$ is a direct sum of several copies of this module category
(the number of indecomposable summands in $\M$ equals the number of irreducible
summands in $\be_\bc$). 
The same argument applies to the multi-fusion category 
$\Vec_{\BZ/2\BZ}^\omega \bt \Coh(Y'\times Y')$ (in this case the irreducible summands of the
unit object are naturally labeled by the elements of the set $Y'$). Thus choosing the set $Y'$
of appropriate size we get
$$\cP^\bc \simeq \Fun_{\Vec_{\BZ/2\BZ}^\omega}(\M,\M)\simeq \Vec_{\BZ/2\BZ}^\omega \bt \Coh(Y'\times Y')$$
and Theorem \ref{main} is proved.

\begin{remark} In \cite[Section 4]{BFO2} the category $\cP_G^\bc$ is endowed with tensor
structure; moreover \cite[Theorem 5.3]{BFO2} states that 
$\cP_G^\bc \simeq \Z(\cP^\bc)$ and the functor $\Gamma^\bc$ is isomorphic
to the induction functor $\Ind : \cP^\bc \to \Z(\cP^\bc)\simeq \cP_G^\bc$
(one can use it to give a shorter proof of Theorem \ref{main}). This implies that for an exceptional
cell $\bc$ we have $\cP^\bc \simeq \Z(\Vec_{\BZ/2\BZ}^\omega)$ (see \cite[Corollary 5.4(b)]{BFO2}
for the case when $\bc$ is not exceptional).
\end{remark}

\bibliographystyle{ams-alpha}

\end{document}